\documentclass[adraft]{eptcs}
\pdfoutput=1
\usepackage[utf8]{inputenc}
\usepackage[T1]{fontenc}
\usepackage[shortlabels]{enumitem}
\usepackage{microtype}
\usepackage[color=lightgray]{todonotes}

\title{Representing Knowledge and Querying Data \\
  using Double-Functorial Semantics}
\author{Michael Lambert
\institute{University of Massachusetts-Boston}
\email{michael.james.lambert@gmail.com}
\and
Evan Patterson
\institute{Topos Institute}
\email{evan@epatters.org}
}
\newcommand{\titlerunning}{Representing Knowledge and Querying Data
  using Double-Functorial Semantics}
\newcommand{\authorrunning}{Lambert and Patterson}

\usepackage{double-cat-macros}
\numberwithin{equation}{section}

\hypersetup{breaklinks,bookmarksnumbered,
  pdftitle={\titlerunning},pdfauthor={\authorrunning}}
\usepackage[capitalize,noabbrev,nameinlink]{cleveref}

\usepackage[style=eptcs]{biblatex}
\begin{document}
\maketitle

\begin{abstract}
  Category theory offers a mathematical foundation for knowledge representation
  and database systems. Popular existing approaches model a database instance as
  a functor into the category of sets and functions, or as a 2-functor into the
  2-category of sets, relations, and implications. The functional and relational
  models are unified by double functors into the double category of sets,
  functions, relations, and implications. In an accessible, example-driven
  style, we show that the abstract structure of a `double category of relations'
  is a flexible and expressive language in which to represent knowledge, and we
  show how queries on data in the spirit of Codd's relational algebra are
  captured by double-functorial semantics.
\end{abstract}

\section{Introduction}
\label{section:introduction}

Knowledge representation and databases are among the most successful
applications of category theory, supported not only by several decades of
theoretical development
\cite{johnson2002,piessens1995,spivak2012fdm,spivak2012ologs} but also by
substantial practical and industrial deployments
\cite{spivak2015,patterson2022}. The general scheme is that the \emph{ontology}
or \emph{database schema} is a small category, possibly equipped with extra
structure, and \emph{instance data} is a structure-preserving functor out of the
schema into a category such as $\Set$. Under this scheme, knowledge
representation and databases are two sides of the same coin: knowledge
representation emphasizes rich schemas and expressive schema languages
(categorical structure), whereas databases center the data instances and
transformations between them. We will pass seamlessly between the terminologies
of the two fields, as the distinction is artificial.

In Spivak's category-theoretic model of a database \cite{spivak2012fdm}, a
\emph{schema} is nothing more than a small category and an \emph{instance} is a
set-valued functor. With an eye toward knowledge representation, Spivak and Kent
explored a richer language \cite{spivak2012ologs}, defining an \emph{ontology
  log}, or \emph{olog}, to be a finite-limit, finite-colimit sketch and an
\emph{instance} to be a set-valued model of the sketch. In both cases the
language is ``functional,'' as arrows in the schema are interpreted as functions
between sets. Since much technology in knowledge representation, including the
Web Ontology Language (OWL), favors relations over functions, the last author
proposed an alternative notion of \emph{relational olog} \cite{patterson2017},
taking Carboni and Walters' \emph{`bicategories of relations'} as the basic
structure \cite{carboni1987}. An ontology now becomes a small `bicategory of
relations' and an instance is a structure-preserving 2-functor from it into
$\bicat{Rel}$. In other words, morphisms in the schema are now interpreted as
relations; also, there are 2-cells interpreted as implications.

We introduce double-categorical ologs, taking the first author's \emph{`double
  categories of relations'} as the foundational structure \cite{lambert2022}. In
outline, a \emph{double olog} will be a small double category, assumed locally
posetal for simplicity and equipped with certain extra structure, and an
\emph{instance} will be a structure-preserving double functor into $\Rel$, the
double category of sets, functions, relations, and implications. As an obvious
first benefit, this approach recognizes that functions and relations are
fundamental concepts and grants them both first-class status. From the
perspective of categorical logic, a double olog possesses all four elements of a
logic fibered over a type theory \cite{jacobs1999}: its objects are types and
its arrows are terms; fibered over those, its proarrows are predicates and its
cells are judgments or implications. By contrast, a functional olog is missing
predicates and judgments, whereas a relational olog is missing terms.

Using the rich structure of a `double category of relations', queries on the
instance data of a double olog can be formulated internally to the schema,
rather than through external mappings. In Spivak's minimal model of databases
\cite[\S{5.3}]{spivak2012fdm}, queries are regarded as particular cases of the
adjoint triple of data migration functors induced by a schema mapping. This
approach, while mathematically economical, is indirect and requires delicate
analysis to reduce to SQL primitives \cite{spivak2015}. By contrast, the basic
operations of Codd's relational algebra \cite[\S{2.1}]{codd1970}---permutations,
projections, joins---appear directly as operations in a `double category of
relations', namely as restrictions or extensions along symmetries, projections,
or diagonals. Thus, as we show in \cref{section:querying}, queries can be
formulated as abstract relations in the schema and then evaluated by the double
functor defining the instance. As a slogan, we say that \emph{querying is
  double-functorial semantics}.

In this paper, we will explain the features of double ologs in an accessible,
example-driven style, but for the benefit of readers familiar with double
category theory, we briefly state the main technical definitions. Further
background can be found in \cref{section:background}. A \define{`double category
  of relations'} is a locally posetal cartesian equipment satisfying a
discreteness or Frobenius axiom \cite[\mbox{Definition 3.2}]{lambert2022}. We
define a \define{double olog}, or just an \define{olog} for short, to be a small
`double category of relations', which we sometimes also assume to have
tabulators. An \define{instance} of a double olog $\dbl{D}$ is a cartesian,
strict double functor $\dbl{D} \to \Rel$, which must also preserve tabulators
whenever $\dbl{D}$ is assumed to have them.

\paragraph{Acknowledgments} Patterson was supported by the Air Force Office of
Scientific Research (AFOSR) Young Investigator Program (YIP) through Award
FA9550-23-1-0133.

\section{Representing Knowledge}

The two fundamental ingredients of a functional olog \cite{spivak2012ologs} are
concepts (aka, types or objects) and attributes (aka, arrows or morphisms),
schematized for example as
    \[\begin{tikzcd}
      {\fbox{Sam Carter}} && {\fbox{person}}
      \arrow["{\text{is a}}", from=1-1, to=1-3]
    \end{tikzcd},\]
where the domain stands for a concept of \emph{being Sam Carter}, the codomain 
represents the concept of \emph{personhood}, and the arrow expresses that 
personhood is an attribute of Sam Carter.\footnote{All of our examples are inspired by
the television show Stargate SG-1, knowledge of which should enhance the reader's 
enjoyment, if not their understanding, of the paper.} This section will show how
double categories---and in particular cartesian equipments and `double 
categories of relations'---can enhance the expressive power of ologs.

\subsection{Expressing Facts}
\label{sec:facts}

A \define{fact} in an olog is defined by Spivak and Kent \cite{spivak2012ologs}
to be a commutative diagram. In other words, a fact is a statement that a pair
of sequences of attributes have the same composite. Suppose we wish to assert
the fact that Frank Simmons is User 4574. In a functional olog, we could write
  
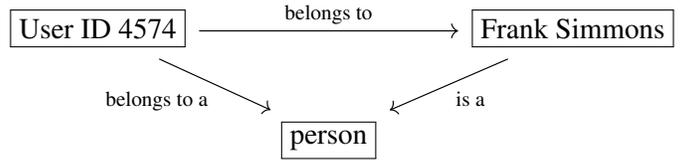
\[\begin{tikzcd}
    {\fbox{User ID 4574}} && {\fbox{Frank Simmons}} \\
    & {\fbox{person}}
    \arrow["{\text{belongs to}}", from=1-1, to=1-3]
    \arrow["{\text{belongs to a}}"', from=1-1, to=2-2]
    \arrow["{\text{is a}}", from=1-3, to=2-2]
  \end{tikzcd},\]
which says that Frank Simmons is the person to whom the ID `User 4574'
belongs.

However, something seems wrong about encoding individual people as types or
objects in a categorical structure. Ordinarily, in first-order logic and type
theory, individual entities are specified in a signature as constants of a
certain type. In category theory, constants are interpreted as \emph{global
  elements}, that is, as arrows $c\colon 1\to A$ from a chosen terminal object to the
type $A$. Following this approach, we would rather fix a terminal object and
posit two individual constants.

We start with an olog asserting that every System Lord has a First Prime,
foremost among Jaffa warriors belonging to that System Lord. This could be 
expressed by a function
    
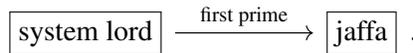
\[\begin{tikzcd}
      {\fbox{system lord}} && {\fbox{jaffa}}
      \arrow["{\text{first prime}}", from=1-1, to=1-3]
    \end{tikzcd}.\]
Let's posit two individuals, Teal'c and Apophis, whom we include in the
olog as constants
    
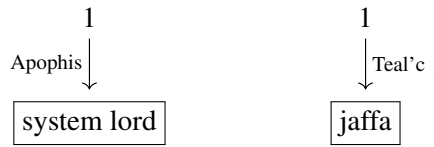
\[\begin{tikzcd}
      1 && 1 \\
      {\fbox{system lord}} && {\fbox{jaffa}}
      \arrow["{\text{Apophis}}"', from=1-1, to=2-1]
      \arrow["{\text{Teal'c}}", from=1-3, to=2-3]
    \end{tikzcd}.\]
The fact that Teal'c is (or rather was) the First Prime of Apophis can be 
expressed as above by a commutative diagram
    
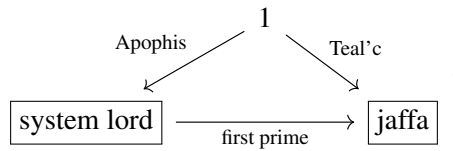
\[\begin{tikzcd}
      & 1 \\
      {\fbox{system lord}} && {\fbox{jaffa}}
      \arrow["{\text{Apophis}}"', from=1-2, to=2-1]
      \arrow["{\text{first prime}}"', from=2-1, to=2-3]
      \arrow["{\text{Teal'c}}", from=1-2, to=2-3]
    \end{tikzcd}.\]

Now suppose we'd like to take a more nuanced view of the relationship ``first
prime.'' The phrasing above carries the assumption that to each System
Lord we can assign a unique First Prime, which is arguably true at a fixed
moment in time. But what if at that moment in time a System Lord is
between First Primes, so that in fact we have have a partial function? Or
perhaps we would like our data to include all of the First Primes a System Lord
has ever had (a one-to-many relationship), or similarly we'd like to consider
all of the System Lords a given Jaffa has served as First Prime. To accommodate
these possibilities, we will alter the olog to specify a \emph{relation} of
being First Prime, along with two constants:

    
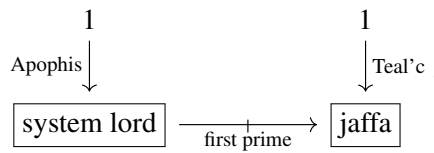
\[\begin{tikzcd}
      1 && 1 \\
      {\fbox{system lord}} && {\fbox{jaffa}}
      \arrow["{\text{Apophis}}"', from=1-1, to=2-1]
      \arrow["{\text{first prime}}"', "\shortmid"{marking}, from=2-1, to=2-3]
      \arrow["{\text{Teal'c}}", from=1-3, to=2-3]
    \end{tikzcd}.\]
Restricting along the Apophis constant creates the cell
    
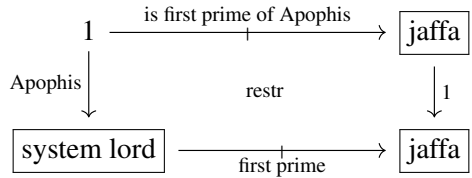
\[\begin{tikzcd}
      1 &&& {\fbox{jaffa}} \\
      {\fbox{system lord}} &&& {\fbox{jaffa}}
      \arrow["{\text{Apophis}}"', from=1-1, to=2-1]
      \arrow[""{name=0, anchor=center, inner sep=0}, "{\text{first prime}}"', "\shortmid"{marking}, from=2-1, to=2-4]
      \arrow["1", from=1-4, to=2-4]
      \arrow[""{name=1, anchor=center, inner sep=0}, "{\text{is first prime of Apophis}}", "\shortmid"{marking}, from=1-1, to=1-4]
      \arrow["{\mathrm{restr}}"{description}, draw=none, from=1, to=0]
    \end{tikzcd}\]
whose image under instance data valued in $\Rel$ would amount to a list of all 
of Jaffa having served Apophis as first prime. How can we express the fact that 
Teal'c is (or was at some point) a first prime of Apophis? This is a matter of 
asking that another restriction factor through (or rather be equal to) the 
truth value $\top$, which is the terminal object in the external
hom-category on $1$. That is, we first form the restriction along both constants
    
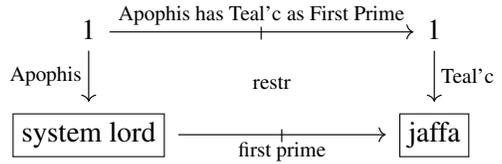
\[\begin{tikzcd}
      1 &&& 1 \\
      {\fbox{system lord}} &&& {\fbox{jaffa}}
      \arrow["{\text{Apophis}}"', from=1-1, to=2-1]
      \arrow[""{name=0, anchor=center, inner sep=0}, "{\text{first prime}}"', "\shortmid"{marking}, from=2-1, to=2-4]
      \arrow["{\text{Teal'c}}", from=1-4, to=2-4]
      \arrow[""{name=1, anchor=center, inner sep=0}, "{\text{Apophis has Teal'c as First Prime}}", "\shortmid"{marking}, from=1-1, to=1-4]
      \arrow["{\text{restr}}"{description}, draw=none, from=1, to=0]
    \end{tikzcd}.\]
Denoting our double olog by $\dbl{D}$, we have the truth value $\top\colon 1 \proto 1$,
the terminal object of $\dbl{D}(1,1)$. The fact that Apophis
has (or had) Teal'c as a First Prime is just asking that the restricted proarrow
above be equal to the local terminal:
    \[\begin{tikzcd}
      1 &&&&& 1 & {=} & \top
      \arrow["{\text{Apophis has Teal'c as First Prime}}", "\shortmid"{marking}, from=1-1, to=1-6]
    \end{tikzcd}.\]

This assertion cannot be so naturally expressed in either functional or
relational ologs.
In a functional olog, we must treat the relation of being First Prime as
functional, which is too restrictive. On the other hand, in a relational olog,
we can only introduce a constant through a relation $c: 1 \proto A$ along with
side equations making the relation into a \emph{map} in the `bicategory of
relations' \cite[\mbox{Definition 1.5} and \mbox{Lemma 2.5}]{carboni1987}. We
then rely on the semantics of the olog in $\bicat{Rel}$ to ensure that the
relation $c$ is interpreted as (the graph of) a function. In a double olog,
we work natively with both functions and relations as is convenient.

\subsection{Creating Types Using Tabulators}

Here is an another double olog, expressing a few relationships between persons,
symbiotes, and types of symbiotes, namely, Goa'uld and Tok'ra:
  
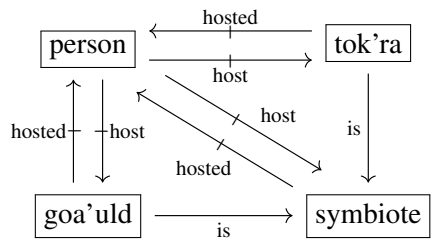
\[\begin{tikzcd}
    {\fbox{person}} && {\fbox{tok'ra}} \\
    \\
    {\fbox{goa'uld}} && {\fbox{symbiote}}
    \arrow["{\text{hosted}}"', "\shortmid"{marking}, shift right=2, from=1-3, to=1-1]
    \arrow["{\text{host}}"', "\shortmid"{marking}, shift right=2, from=1-1, to=1-3]
    \arrow["{\text{hosted}}", "\shortmid"{marking}, shift left=2, from=3-1, to=1-1]
    \arrow["{\text{host}}", "\shortmid"{marking}, shift left=2, from=1-1, to=3-1]
    \arrow["{\text{is}}"', from=3-1, to=3-3]
    \arrow["{\text{is}}"', from=1-3, to=3-3]
    \arrow["{\text{hosted}}"{pos=0.4}, "\shortmid"{marking}, shift left=2, shorten <=6pt, from=3-3, to=1-1]
    \arrow["{\text{host}}"{pos=0.6}, "\shortmid"{marking}, shift left=2, shorten <=6pt, from=1-1, to=3-3]
  \end{tikzcd}.\]
Every person is a potential symbiote host. But notice that the relation
$\text{host}\colon \fbox{person} \proto \fbox{symbiote}$ and those valued in
$\fbox{tok'ra}$ and $\fbox{goa'uld}$ are not functional since not every person is
in fact a host. An instance of ``$\text{hosted}$'' is generally many-to-one
since many symbiotes have at least one host. It should be asked in equations or 
as a definition that $\text{host}^\dagger = \text{hosted}$, where the dagger
$(-)^\dagger$ denotes the converse or opposite relation. All Tok'ra are
symbiotes; so are Goa'uld. Thus, the two arrows labeled `$\text{is}$' are
genuinely functional, indicating something like subtypes.

Typically a Goa'uld leaves its host only upon the death of the host, and if
the Goa'uld is forcibly removed, then the host will die due to release of
toxins. So, simplifying slightly, we can make the assumption that each person
is a host at most once. This is expressed by imposing the equation
  \[\begin{tikzcd}
    \id & {\fbox{goa'uld}} && {\fbox{person}} && {\fbox{goa'uld}}
    \arrow["{\text{hosted}}", "\shortmid"{marking}, from=1-2, to=1-4]
    \arrow["{\text{host}}", "\shortmid"{marking}, from=1-4, to=1-6]
    \arrow["{=}"{description}, draw=none, from=1-1, to=1-2]
  \end{tikzcd}\]
asserting that the host and hosted relation compose to the identity
on $\fbox{goa'uld}$. In other words, the relation $\text{host}$ is a \emph{partial map}.
\cref{prop:PartialMapsAreMonic} shows that if the double olog $\dbl{D}$ were
taken to be a unit-pure `double category of relations' with tabulators, then
this assumption  would imply that the tabulator of the host-relation, namely
$\fbox{host-goa'uld pairs}$, is the apex of a span
\[\begin{tikzcd}
	{\fbox{person}} & {\fbox{host-goa'uld pairs}} & {\fbox{goa'uld}}
	\arrow["c", from=1-2, to=1-3]
	\arrow["d"', tail, from=1-2, to=1-1]
\end{tikzcd}\]
where $d$ is monic, that is, a partial map in the underlying category
$\dbl{D}_0$ of objects and arrows.
If the tabulators in $\dbl{D}$ are also assumed to be \emph{strong}, then the
process is reversible, in the sense that a partial map (a span with monic left 
leg) would induce a partial morphism in the relational sense. Thus, under
suitable assumptions, one can pass back and forth between the two possible
representations of a partial map in a double olog.

Even without these extra conditions, the mere presence of
extensions and tabulators increases the expressive power of double ologs
considerably. For example, we can create a type of persons who are or have been
hosts. First, create a proposition for the notion \emph{has hosted} as an 
extension cell
  \[\begin{tikzcd}
    {\fbox{person}} && {\fbox{symbiote}} \\
    {\fbox{person}} && 1
    \arrow[""{name=0, anchor=center, inner sep=0}, "{\text{host}}", "\shortmid"{marking}, from=1-1, to=1-3]
    \arrow["{=}"', from=1-1, to=2-1]
    \arrow["{!}", from=1-3, to=2-3]
    \arrow[""{name=1, anchor=center, inner sep=0}, "{\text{has hosted}}"', "\shortmid"{marking}, from=2-1, to=2-3]
    \arrow["{\text{ext}}"{description}, draw=none, from=0, to=1]
  \end{tikzcd}.\]
This is interpreted as the image of the host relation in the type of persons.
Then extract the comprehension or subobject classified by this proposition by
taking its tabulator. This gives us a type of hosts:
  \[\begin{tikzcd}
    & {\fbox{host}} \\
    {\fbox{person}} && 1
    \arrow[""{name=0, anchor=center, inner sep=0}, "{\text{has hosted}}"', "\shortmid"{marking}, from=2-1, to=2-3]
    \arrow["{\text{is}}"', from=1-2, to=2-1]
    \arrow["{!}", from=1-2, to=2-3]
    \arrow["{\text{tab}}"{description}, draw=none, from=1-2, to=0]
  \end{tikzcd}.\]
Likewise, there is a type of hosted symbiotes constructed by the extension on
the left followed by the tabulator on the right:
\begin{equation*}
  \begin{tikzcd}
    {\fbox{person}} && {\fbox{symbiote}} \\
    1 && {\fbox{symbiote}}
    \arrow["{!}"', from=1-1, to=2-1]
    \arrow[""{name=0, anchor=center, inner sep=0}, "{\text{host}}", "\shortmid"{marking}, from=1-1, to=1-3]
    \arrow["{=}", from=1-3, to=2-3]
    \arrow[""{name=1, anchor=center, inner sep=0}, "{\text{has been hosted}}"', "\shortmid"{marking}, from=2-1, to=2-3]
    \arrow["{\text{ext}}"{description}, draw=none, from=0, to=1]
  \end{tikzcd}
  \quad\leadsto\quad
  \begin{tikzcd}
    & {\fbox{hosted symbiote}} \\
    1 && {\fbox{symbiote}}
    \arrow["{!}"', from=1-2, to=2-1]
    \arrow[""{name=0, anchor=center, inner sep=0}, "{\text{has been hosted}}"', "\shortmid"{marking}, from=2-1, to=2-3]
    \arrow["{\text{is}}", from=1-2, to=2-3]
    \arrow["{\text{tab}}", shift right, draw=none, from=1-2, to=0]
  \end{tikzcd}.
\end{equation*}
The arrow on the right again represents a kind of subtype since there are
symbiotes such as queens who are not always hosted and symbiotes such as
prim'ta who are never hosted for whatever reason.
In short, tabulators enable propositions, and relations generally, to be reified 
as subtypes of the original types.

\subsection{Forming Conjunctions Using Local Products}
\label{subsection:conjunction}

What if we want to assert as a fact that two propositions both hold or that they
are equal? That can be done using \emph{local products}
(\cref{subsection:Appendixlocalproducts}).

Let's look again at an example. Consider the olog with four relations:
    
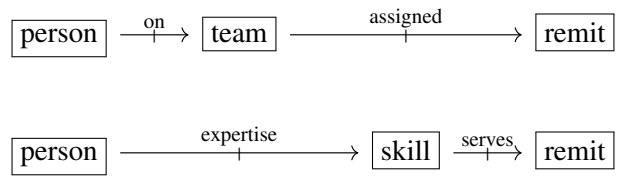
\[\begin{tikzcd}
      {\fbox{person}} & {\fbox{team}} && {\fbox{remit}} \\
      {\fbox{person}} && {\fbox{skill}} & {\fbox{remit}}
      \arrow["{\text{on}}", "\shortmid"{marking}, from=1-1, to=1-2]
      \arrow["{\text{assigned}}", "\shortmid"{marking}, from=1-2, to=1-4]
      \arrow["{\text{expertise}}", "\shortmid"{marking}, from=2-1, to=2-3]
      \arrow["{\text{serves}}", "\shortmid"{marking}, from=2-3, to=2-4]
    \end{tikzcd}.\]
A fairly complex hypothetical underlies this olog,
captured by the question of whether the expertise of given persons serves the
remit of the teams of which they are members. We could require by fiat an 
equality between these objects in the appropriate hom-category. In some sense 
this \emph{should be the case} if the teams have been formed correctly.
But perhaps that's not necessarily the case, and we'd rather ask about
\emph{the extent to which} our question is answered in the affirmative. We can
do this using a local product, which, in the example at hand, looks like
    
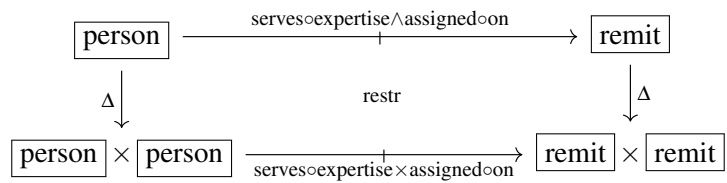
\[\begin{tikzcd}
      {\fbox{person}} &&&& {\fbox{remit}} \\
      {\fbox{person}\times \fbox{person}} &&&& {\fbox{remit}\times \fbox{remit}}
      \arrow[""{name=0, anchor=center, inner sep=0}, "{\text{serves}\circ\text{expertise}\times \text{assigned}\circ\text{on}}"', "\shortmid"{marking}, from=2-1, to=2-5]
      \arrow[""{name=1, anchor=center, inner sep=0}, "{\text{serves}\circ\text{expertise}\wedge\text{assigned}\circ\text{on}}", "\shortmid"{marking}, from=1-1, to=1-5]
      \arrow["\Delta"', from=1-1, to=2-1]
      \arrow["\Delta", from=1-5, to=2-5]
      \arrow["{\text{restr}}"{description}, draw=none, from=1, to=0]
    \end{tikzcd}.\]
Let's postulate some concrete data to see what this local product looks like.
Define a data instance $\dbl{D}\to\Rel$ for our double olog $\dbl{D}$ by
assigning the types to the data columns
    \[
      \begin{minipage}{1.25in}
        \begin{tabular}{| l |}
          \hline {\bf person}\\
          \hline Carter\\
          \hline Connor\\
          \hline Jackson\\
          \hline Kovacek\\
          \hline Teal'c\\
          \hline
        \end{tabular}  
      \end{minipage}
      \begin{minipage}{1.25in}
        \begin{tabular}{| l |}
          \hline {\bf team}\\
          \hline SG1\\
          \hline SG3\\
          \hline SG9\\
          \hline SG11\\
          \hline
        \end{tabular}  
      \end{minipage}
      \begin{minipage}{1.5in}
        \begin{tabular}{| l |}
          \hline {\bf skill}\\
          \hline archaeology\\
          \hline anthropology\\
          \hline astrophysics\\
          \hline engineering\\
          \hline law\\
          \hline philology\\
          \hline tracking\\
          \hline weapons\\
          \hline
        \end{tabular}  
      \end{minipage}
      \begin{minipage}{2in}
        \begin{tabular}{| l |}
          \hline {\bf remit}\\
          \hline combat\\
          \hline diplomacy\\
          \hline exploration\\
          \hline science \& engineering\\
          \hline search \& rescue\\
          \hline
        \end{tabular}\;\;.
      \end{minipage}
    \]
The first two relations are then the following:
    \[
      \begin{minipage}{2in}
        \begin{tabular}{| l | l |}
          \hline\multicolumn{2}{| c |}{\bf on}\\
          \hline {\bf person} & {\bf team}\\
          \hline Carter & SG1\\
          \hline Connor & SG9\\
          \hline Jackson & SG1\\
          \hline Kovacek & SG9\\
          \hline Teal'c & SG1\\
          \hline Teal'c & SG3\\
          \hline
        \end{tabular}  
      \end{minipage}
      \begin{minipage}{2.5in}
        \begin{tabular}{| l | l |}
          \hline\multicolumn{2}{| c |}{\bf assigned}\\
          \hline {\bf team} & {\bf remit}\\
          \hline SG1 & exploration\\
          \hline SG3 & combat\\
          \hline SG3 & search \& rescue\\
          \hline SG9 & diplomacy\\
          \hline SG11 & science \& engineering\\
          \hline
        \end{tabular}  
      \end{minipage}
    \]
and the remaining two are the big tables filled with all of Daniel Jackson's PhDs:
    \[
      \begin{minipage}{2.5in}
      \begin{tabular}{| l | l |}
        \hline\multicolumn{2}{| c |}{\bf expertise}\\
        \hline {\bf person} & {\bf skill}\\
        \hline Carter & astrophysics\\
        \hline Carter & weapons\\
        \hline Connor & engineering\\
        \hline Jackson & anthropology\\
        \hline Jackson & archaeology\\
        \hline Jackson & philology\\
        \hline Kovacek & law\\
        \hline Teal'c & weapons\\
        \hline Teal'c & tracking\\
        \hline
      \end{tabular}
      \end{minipage}
      \begin{minipage}{3in}
        \begin{tabular}{| l | l |}
          \hline\multicolumn{2}{| c |}{\bf serves}\\
          \hline {\bf skill} & {\bf remit}\\
          \hline astrophysics & exploration\\
          \hline anthropology & exploration\\
          \hline archaeology & exploration\\
          \hline engineering & science \& engineering\\
          \hline law & diplomacy\\
          \hline philology & exploration\\
          \hline weapons & combat\\
          \hline weapons & exploration\\
          \hline tracking & search \& rescue\\
          \hline
        \end{tabular}\;\;.  
      \end{minipage}
    \]
\emph{None} of these relations are functional. Perhaps some of the
associations in the $\textbf{serves}$ relation are debatable and certainly we
haven't included all skills possessed by all the individuals in the $\textbf
{person}$ table, but the point is just to show how local products work.
One can trace through the
construction as a restriction first in $\dbl{D}$ and then in $\Rel$. But the
local products formula in \cref{equation:localproductsinRelformula}
above makes the computation easy. The local product is just the set of all
$\fbox{person}$-$\fbox{remit}$ pairs related under \emph{both} composites
$\text{serves}\circ\text{expertise}$ and $\text{assigned}\circ\text{on}$. To be
explicit, we can compute the two composites as
  \[
    \begin{minipage}{3in}
    \begin{tabular}{| l | l |}
      \hline\multicolumn{2}{| c |}{\bf serves $\circ$ expertise}\\
      \hline {\bf person} & {\bf remit}\\
      \hline Carter & exploration\\
      \hline Carter & combat\\
      \hline Connor & science \& engineering\\
      \hline Jackson & exploration\\
      \hline Kovacek & diplomacy\\
      \hline Teal'c & combat\\
      \hline Teal'c & search \& rescue\\
      \hline
    \end{tabular}
  \end{minipage}
  \begin{minipage}{2.5in}
    \begin{tabular}{| l | l |}
      \hline\multicolumn{2}{| c |}{\bf assigned $\circ$ on}\\
      \hline {\bf person} & {\bf remit}\\
      \hline Carter & exploration\\
      \hline Connor & diplomacy\\
      \hline Jackson & exploration\\
      \hline Kovacek & diplomacy\\
      \hline Teal'c & exploration\\
      \hline Teal'c & combat\\
      \hline Teal'c & search \& rescue\\
      \hline
    \end{tabular}\;\;.
  \end{minipage}
  \]
In each case we've used the formula for composition of relations using 
existential quantification. Notice that despite Laurence Connor's training in 
aerospace engineering, he serves on the diplomacy team SG9, while Sam Carter's
weapons training would qualify her for combat, but she is instead a member
of exploration team SG1. Thus, the local product is the table 
excluding these two entries, namely
  \[
    \begin{tabular}{| l | l |}
      \hline\multicolumn{2}{| c |}{\bf serves $\circ$ expertise $\wedge$ assigned $\circ$ on}\\
      \hline {\bf person} & {\bf remit}\\
      \hline Carter & exploration\\
      \hline Jackson & exploration\\
      \hline Kovacek & diplomacy\\
      \hline Teal'c & combat\\
      \hline Teal'c & search \& rescue\\
      \hline
    \end{tabular}\;\;.
  \]
The table thus gives precisely those people whose expertise serves the
purpose of the teams to which they are assigned as well as the particular 
remit. 

\section{Querying Data}
\label{section:querying}

The conceit of this section is that `double categories of relations' provide a
simple and satisfying framework for querying data. In the approach based on
functorial data migration \cite{spivak2012fdm}, queries are handled externally
through the device of adjoint functors between copresheaf categories and
regarded as a special case of data migration. We will see here that queries can
be expressed internally to the schema using the operations available in a
`double category of relations', or more generally in a cartesian equipment. As a
slogan: \emph{querying is double-functorial semantics}.

\subsection{Select}

The select operation is perhaps the easiest to describe in our
double-categorical formalism. Selection is performed by taking an extension
along projection morphisms.

Suppose we are describing the attributes belonging to a concept of a mission. In
a functional olog \cite{spivak2012ologs}, we would schematize the concept with
functional relations indicated by the arrows identifying the leaf tables, such
as:
    
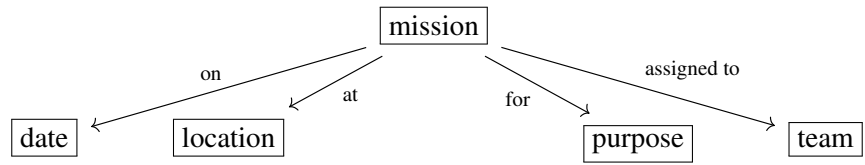
\[\begin{tikzcd}
      && {\fbox{mission}} \\
      {\fbox{date}} & {\fbox{location}} && {\fbox{purpose}} & {\fbox{team}}
      \arrow["{\text{on}}"', from=1-3, to=2-1]
      \arrow["{\text{at}}", from=1-3, to=2-2]
      \arrow["{\text{for}}"', from=1-3, to=2-4]
      \arrow["{\text{assigned to}}", from=1-3, to=2-5]
    \end{tikzcd}.\]
However, this multi-span is just an encoding of a relation, which we choose to
present double-categorically as a quaternary relation:
    \[\begin{tikzcd}
      {\fbox{date}\times \fbox{location}} && {\fbox{purpose}\times\fbox{team}}
      \arrow["{\text{mission}}", "\shortmid"{marking}, from=1-1, to=1-3]
    \end{tikzcd}.\]
Of course, we could have chosen to partition the four types differently, but
that would not meaningfully alter the result.
We display the instance data for the relation in the conventional style as a
multi-column table:
    \[
      \begin{tabular}{|l|l|l|l|}
        \hline\multicolumn{4}{|c|}{\bf mission} \\
        \hline {\bf date} & {\bf location} & {\bf purpose} & {\bf team} \\
        \hline 02-06-1998 & P41-771 & search \& rescue & SG3 \\
        \hline 07-31-1998 & Cimmeria & assist Cimmerians & SG1 \\
        \hline 01-02-1999 & P3R-272 & investigate inscriptions & SG1 \\
        \hline 10-22-1999 & Ne'tu & search \& rescue & SG1 \\
        \hline 08-06-2004 & Tegalus & negotiation & SG9 \\
        \hline
      \end{tabular}\;\;.
\]
We'll assume in the background that the types have been instanced with data, but
it's not necessary to say at this point what the data actually is.

Selecting columns is now simply a matter of extending the relation along
projections. For example, say we want only to remember the dates that teams went
on missions but we care neither what the location nor purpose of the missions
were. In this case, we take an extension along the projection morphisms as
shown:
    
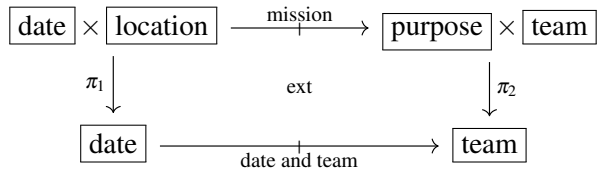
\[\begin{tikzcd}
      {\fbox{date}\times \fbox{location}} && {\fbox{purpose}\times\fbox{team}} \\
      {\fbox{date}} && {\fbox{team}}
      \arrow[""{name=0, anchor=center, inner sep=0}, "{\text{mission}}", "\shortmid"{marking}, from=1-1, to=1-3]
      \arrow["{\pi_1}"', from=1-1, to=2-1]
      \arrow["{\pi_2}", from=1-3, to=2-3]
      \arrow[""{name=1, anchor=center, inner sep=0}, "{\text{date and team}}"', "\shortmid"{marking}, from=2-1, to=2-3]
      \arrow["{\text{ext}}"{description}, draw=none, from=0, to=1]
    \end{tikzcd}.\]
In the double category $\Rel$, the extension is computed by taking an image.
That is, a date and team are in the extension relation if, and only if, they are
related (with two other elements of some $4$-tuple) in the original relation.
In this case, we end up with a binary relation instanced by the table
    \[
      \begin{tabular}{|l|l|}
        \hline\multicolumn{2}{|c|}{\bf date and team} \\
        \hline {\bf date} & {\bf team} \\
        \hline 02-06-1998 & SG3 \\
        \hline 07-31-1998 & SG1 \\
        \hline 01-02-1999 & SG1 \\
        \hline 10-22-1999 & SG1 \\
        \hline 08-06-2004 & SG9 \\
        \hline
      \end{tabular}\;\;.
    \]
Alternatively, we might ask just for the occasions on which teams have visited 
locations disregarding the purpose of the mission. This would be accomplished 
by an extension yielding a ternary relation with accompanying table:
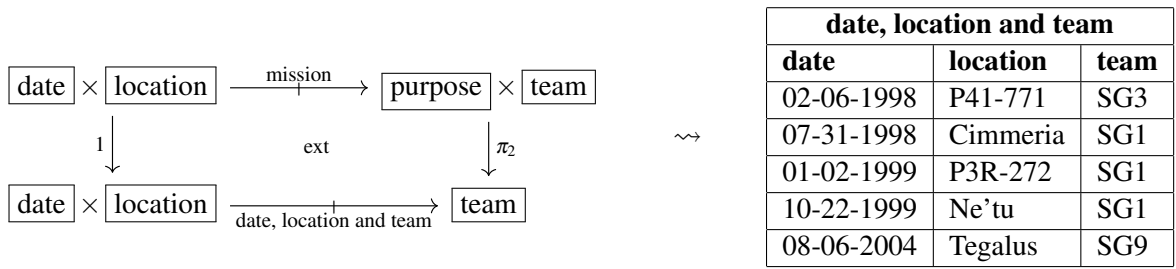
\[
  \begin{tikzcd}
    {\fbox{date}\times \fbox{location}} && {\fbox{purpose}\times\fbox{team}} \\
    {\fbox{date}\times\fbox{location}} && {\fbox{team}}
    \arrow[""{name=0, anchor=center, inner sep=0}, "{\text{mission}}", "\shortmid"{marking}, from=1-1, to=1-3]
    \arrow["1"', from=1-1, to=2-1]
    \arrow["{\pi_2}", from=1-3, to=2-3]
    \arrow[""{name=1, anchor=center, inner sep=0}, "{\text{date, location and team}}"', "\shortmid"{marking}, from=2-1, to=2-3]
    \arrow["{\text{ext}}"{description}, draw=none, from=0, to=1]
  \end{tikzcd}
  \qquad\leadsto\qquad
  \begin{tabular}{|l|l|l|}
    \hline\multicolumn{3}{|c|}{\bf date, location and team} \\
    \hline {\bf date} & {\bf location} & {\bf team} \\
    \hline 02-06-1998 & P41-771 & SG3 \\
    \hline 07-31-1998 & Cimmeria & SG1 \\
    \hline 01-02-1999 & P3R-272 & SG1 \\
    \hline 10-22-1999 & Ne'tu & SG1 \\
    \hline 08-06-2004 & Tegalus & SG9 \\
    \hline
  \end{tabular}\;\;.
\]
Finally, we might want to know only the date on which a location was visited.
This would be an extension using the terminal object on one side:
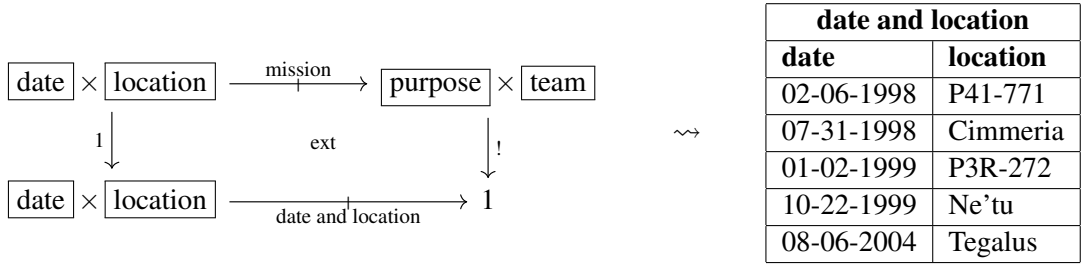
\[
  \begin{tikzcd}
    {\fbox{date}\times \fbox{location}} && {\fbox{purpose}\times\fbox{team}} \\
    {\fbox{date}\times\fbox{location}} && 1
    \arrow[""{name=0, anchor=center, inner sep=0}, "{\text{mission}}", "\shortmid"{marking}, from=1-1, to=1-3]
    \arrow["1"', from=1-1, to=2-1]
    \arrow["!", from=1-3, to=2-3]
    \arrow[""{name=1, anchor=center, inner sep=0}, "{\text{date and location}}"', "\shortmid"{marking}, from=2-1, to=2-3]
    \arrow["{\text{ext}}"{description}, draw=none, from=0, to=1]
  \end{tikzcd}
  \qquad\leadsto\qquad
  \begin{tabular}{|l|l|}
    \hline\multicolumn{2}{|c|}{\bf date and location} \\
    \hline {\bf date} & {\bf location} \\
    \hline 02-06-1998 & P41-771 \\
    \hline 07-31-1998 & Cimmeria \\
    \hline 01-02-1999 & P3R-272 \\
    \hline 10-22-1999 & Ne'tu \\
    \hline 08-06-2004 & Tegalus \\
    \hline
  \end{tabular}\;\;.
\]
We take these examples to be sufficient evidence that the select operation
can be performed by extending along projection morphisms. The point is simply to
use whichever projections return the desired columns.

\subsection{Filter}

The operation of \emph{filtering} data asks for rows that satisfy a certain
property, in the simplest case that a particular attribute has a specific value.
This is one of the last topics of \cite{spivak2012fdm}, treated in \S 5.3 in an
example using adjoint functors arising from slices of copresheaf toposes. In a
double olog, filtering can be performed using restrictions provided we allow
ourselves types in the olog isolating the attribute value that we're looking
for.

Suppose we wish to filter the previous data for just the missions assigned to
the flagship team SG1. As explained in \cref{sec:facts}, we can include the
individual team SG1 in our olog in two different ways. First, we could add an
arrow
  \[\begin{tikzcd}
    {\fbox{SG1}} & {\fbox{team}}
    \arrow["{\text{is a}}", from=1-1, to=1-2]
  \end{tikzcd}.\]
On the other hand, we could stipulate that SG1 is rather the name of a global 
element of the type, namely, an arrow from the terminal
  \[\begin{tikzcd}
    1 & {\fbox{team}}
    \arrow["{\text{SG1}}", from=1-1, to=1-2]
  \end{tikzcd}.\]
Both approaches suffice to filter the data but return slightly
different arrangements. In the former case, we take the restriction
  
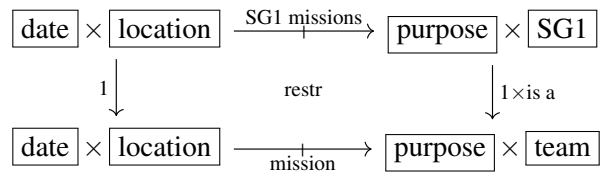
\[\begin{tikzcd}
    {\fbox{date}\times \fbox{location}} && {\fbox{purpose}\times\fbox{SG1}} \\
    {\fbox{date}\times \fbox{location}} && {\fbox{purpose}\times\fbox{team}}
    \arrow[""{name=0, anchor=center, inner sep=0}, "{\text{mission}}"', "\shortmid"{marking}, from=2-1, to=2-3]
    \arrow["1"', from=1-1, to=2-1]
    \arrow["{1\times\text{is a}}", from=1-3, to=2-3]
    \arrow[""{name=1, anchor=center, inner sep=0}, "{\text{SG1 missions}}", "\shortmid"{marking}, from=1-1, to=1-3]
    \arrow["{\text{restr}}"{description}, draw=none, from=1, to=0]
  \end{tikzcd}\]
and the data returned is a subtable consisting only of certain
rows of the original \textbf{mission} table, namely
  \[
    \begin{tabular}{|l|l|l|l|}
      \hline\multicolumn{4}{|c|}{\bf SG1 missions} \\
      \hline {\bf date} & {\bf location} & {\bf purpose} & {\bf team} \\ 
      \hline 07-31-1998 & Cimmeria & assist Cimmerians & SG1 \\
      \hline 01-02-1999 & P3R-272 & investigate inscriptions & SG1 \\
      \hline 10-22-1999 & Ne'tu & search \& rescue & SG1 \\
      \hline
    \end{tabular}\;\;.
  \]
Of course, the last column on the right is superfluous since we are specifically
filtering for that particular value. This redundancy can be avoided using the
global element approach. In this case we have to be especially careful to label 
the restricted relation in the olog. That is, we \emph{define} the relation
``\text{SG1 missions}'' to be the restriction 
  
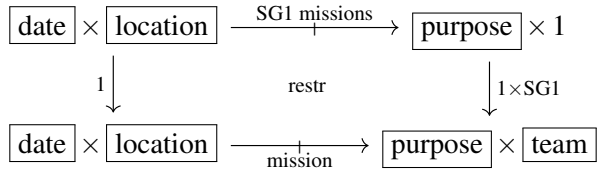
\[\begin{tikzcd}
    {\fbox{date}\times \fbox{location}} && {\fbox{purpose}\times 1} \\
    {\fbox{date}\times \fbox{location}} && {\fbox{purpose}\times\fbox{team}}
    \arrow[""{name=0, anchor=center, inner sep=0}, "{\text{mission}}"', "\shortmid"{marking}, from=2-1, to=2-3]
    \arrow["1"', from=1-1, to=2-1]
    \arrow["{1\times\text{SG1}}", from=1-3, to=2-3]
    \arrow[""{name=1, anchor=center, inner sep=0}, "{\text{SG1 missions}}", "\shortmid"{marking}, from=1-1, to=1-3]
    \arrow["{\text{restr}}"{description}, draw=none, from=1, to=0]
  \end{tikzcd}.\]
This returns the simplified table omitting the redundant column:
  \[
    \begin{tabular}{|l|l|l|}
      \hline\multicolumn{3}{|c|}{\bf SG1 missions} \\
      \hline {\bf date} & {\bf location} & {\bf purpose} \\ 
      \hline 07-31-1998 & Cimmeria & assist Cimmerians \\
      \hline 01-02-1999 & P3R-272 & investigate inscriptions \\
      \hline 10-22-1999 & Ne'tu & search \& rescue \\
      \hline
    \end{tabular}\;\;.
  \]
If in addition we wanted to select certain columns (say we were interested
only in the dates on which SG1 visited certain places), we could use the
projection-extension method above to produce a table with only those
columns. Such a combination is effectively a \emph{select-from-where} query in
SQL, restricted to a single table. We now turn to how such queries can be
extended to multiple tables.

\subsection{Inner Joins}

The inner join operation is an example of local products, already used in
\cref{subsection:conjunction} and reviewed in
\cref{subsection:Appendixlocalproducts}. Let's take a simple olog, namely, a
fragment of the teams olog above with the two relations
    
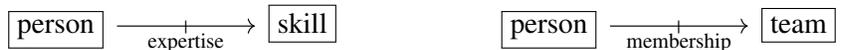
\[\begin{tikzcd}
      {\fbox{person}} && {\fbox{skill}} && {\fbox{person}} && {\fbox{team}}
      \arrow["{\text{expertise}}"', "\shortmid"{marking}, from=1-1, to=1-3]
      \arrow["{\text{membership}}"', "\shortmid"{marking}, from=1-5, to=1-7]
    \end{tikzcd}.\]
These relations are intended to express that a person may have expertise
consisting of training or facility in a particular skill, and that each person
may or may not belong to a particular SGC team. Denote this double olog by $\dbl{D}$.
As instance data $\dbl{D} \to \Rel$, let us assign the following data:
    \[
        \begin{tabular}{| l |}
          \hline {\bf person}\\
          \hline Hammond\\
          \hline Kovacek\\
          \hline Maybourne\\
          \hline Morrison\\
          \hline O'Neill\\
          \hline Rothman\\
          \hline Simmons\\
          \hline Warren\\
          \hline
        \end{tabular}
        \qquad
        \begin{tabular}{| l |}
          \hline {\bf skill}\\
          \hline archaeology\\
          \hline chicanery\\
          \hline combat\\
          \hline command\\
          \hline law\\
          \hline sociopathy\\
          \hline
        \end{tabular}  
        \qquad
        \begin{tabular}{| l |}
          \hline {\bf teams}\\
          \hline SG1\\
          \hline SG3\\
          \hline SG9\\
          \hline SG11\\
          \hline
        \end{tabular}
        \qquad
        \begin{tabular}{| l | l |}
          \hline\multicolumn{2}{| c |}{\bf expertise}\\
          \hline {\bf person} & {\bf skill}\\
          \hline Hammond & command\\
          \hline Kovacek & law\\
          \hline Maybourne & chicanery\\
          \hline Morrison & combat\\
          \hline O'Neill & command\\
          \hline Rothman & archaeology\\
          \hline Simmons & sociopathy\\
          \hline Warren & combat\\
          \hline
        \end{tabular}  
        \qquad
        \begin{tabular}{| l | l |}
          \hline\multicolumn{2}{| c |}{\bf membership}\\
          \hline {\bf person} & {\bf team}\\
          \hline Kovacek & SG9\\
          \hline Morrison & SG3\\
          \hline O'Neill & SG1\\
          \hline Rothman & SG11\\
          \hline Warren & SG3\\
          \hline
        \end{tabular}
    \]
As always, we assume that this data takes the form of a cartesian strict double
functor $\dbl{D}\to \Rel$. In our example, the inner join along
the shared column $\fbox{person}$ is given by the restriction
    \[\begin{tikzcd}
      {\fbox{person}} &&& {\fbox{skill}\times\fbox{team}} \\
      {\fbox{person}\times\fbox{person}} &&& {\fbox{skill}\times\fbox{team}}
      \arrow["\Delta"', from=1-1, to=2-1]
      \arrow[""{name=0, anchor=center, inner sep=0}, "{\text{expertise}\bowtie\text{membership}}", "\shortmid"{marking}, from=1-1, to=1-4]
      \arrow["1", from=1-4, to=2-4]
      \arrow[""{name=1, anchor=center, inner sep=0}, "{\text{expertise}\times\text{membership}}"', "\shortmid"{marking}, from=2-1, to=2-4]
      \arrow["{\text{restr}}"{description}, draw=none, from=0, to=1]
    \end{tikzcd}.\]
In relations, the restriction is computed as a pullback of the product along 
the arrow $\Delta\times 1$. That is, form the cartesian product of all pairs from the two
relations and match the $\fbox{person}$ argument. We obtain the following table
    \[
      \begin{tabular}{|l|l|l|}
        \hline\multicolumn{3}{|c|}{\bf expertise $\bowtie$ membership} \\
        \hline {\bf person} & {\bf skill} & {\bf team} \\
        \hline Kovacek & law & SG9\\
        \hline Morrison & combat & SG3\\
        \hline O'Neill & command & SG1\\
        \hline Rothman & archaeology & SG11\\
        \hline Warren & combat & SG3\\
        \hline
      \end{tabular}
    \]
recording both the skill and team of those persons who have both listing 
without any extra rows or null values. This is precisely the inner join of the 
two tables. Notice also that this approach avoids repeating the person column
in the joined table.

\section{Conclusion}

We have shown by example how the structure of a `double category of relations'
can be used to represent knowledge in an ontology and to express queries against
data instances of the ontology. Progressing from these illustrations to a
working knowledge representation and database system will require further
mathematical and engineering work.

Theoretical questions include what should be the basic double-categorical
structure, including which notions are primitive and which are derived, and how
the structure can be extended to meet practical needs. For instance, we have so
far taken \emph{local products} (conjunction) to be a concept derived from
external products and restrictions, but in a different formulation of
double-categorical products \cite{patterson2024}, local products become
primitive operations, which is possibly more convenient. One unavoidable
practical need is to support \emph{data attributes} valued in fixed types such
as strings or calendar dates. Several approaches to data attributes have been
proposed in the literature on categorical databases
\cite{spivak2012fdm,spivak2015,schultz2017}; at least one of them should be
adapted to double-categorical databases. More ambitiously, it should be
investigated whether our functorial approach to querying can be extended to
support \emph{aggregation} \cite{spivak2021}.

Design considerations for an engineered system have not been addressed. In an
implementation, the mathematical expression of queries by double-functorial
semantics must take the more concrete form of a query language, textual or
graphical. Algorithms for evaluating queries in the language or reducing them to
SQL queries must be devised. To do the latter, one would need to be more precise
about the observation in \cref{section:introduction} that the primitive notions
of a `double category of relations' closely resemble traditional relational
algebra, in contrast to previous approaches to categorical databases. In this
way, as always, engineering challenges may spur further theoretical development.

\printbibliography[heading=bibintoc]

\appendix
\section{Background on Double Category Theory}
\label{section:background}

In this appendix, we review background material on double category theory, as
well as prove a relevant new characterization of partial maps in a `double
category of relations.' A general reference on double category theory is the
textbook by Grandis \cite{grandis2019}.

\subsection{Double Categories of Relations}

Briefly, a \textbf{`double category of relations'} \cite{lambert2022} is a 
double category that is also cartesian and an equipment and that satisfies a 
discreteness condition distinguishing the `bicategories of relations' of 
\cite{carboni1987} from mere cartesian bicategories. For simplicity we also 
assume that any such `double category of relations' is locally posetal. The 
usual double category of relations $\Rel$ is the canonical example. That a 
`double category of relations' $\mathbb D$ is cartesian means that it has 
finite products in an appropriately coherent way \cite{aleiferi2018}, namely, 
the canonical double functors $\mathbb D\to \mathbb D\times\mathbb D$ and 
$\mathbb D\to 1$ have right adjoints in the 2-category of double categories, 
(pseudo)-double functors and transformations. That a `double category of 
relations' is an equipment \cite{shulman2008} means that the canonical 
source-target projection functor $\mathbb D_1\to\mathbb D_0\times\mathbb D_0$ 
is a bifibration.

In more practical terms, the latter bifibration condition is equivalent to the 
fact that niches and coniches can be completed to proper cells 
    \[\begin{tikzcd}
      A & B && A & B \\
      C & D && C & D
      \arrow["f"', from=1-1, to=2-1]
      \arrow[""{name=0, anchor=center, inner sep=0}, "S"', "\shortmid"{marking}, from=2-1, to=2-2]
      \arrow["g", from=1-2, to=2-2]
      \arrow["h"', from=1-4, to=2-4]
      \arrow[""{name=1, anchor=center, inner sep=0}, "R", "\shortmid"{marking}, from=1-4, to=1-5]
      \arrow["k", from=1-5, to=2-5]
      \arrow[""{name=2, anchor=center, inner sep=0}, dashed, from=1-1, to=1-2]
      \arrow[""{name=3, anchor=center, inner sep=0}, dashed, from=2-4, to=2-5]
      \arrow["{\mathrm{restr}}"{description}, draw=none, from=2, to=0]
      \arrow["{\mathrm{ext}}"{description}, draw=none, from=1, to=3]
    \end{tikzcd}\]
that are (respectively) cartesian and opcartesian with respect to the 
projection functor $\mathbb D_1\to\mathbb D_0\times\mathbb D_0$. The cartesian 
cell on the left is thought of as a \emph{restriction} since in $\Rel$ it is 
given by taking a pullback of the relation $S$ along the pair $f\times g$. On 
the other hand, the opcartesian cell on the right is an \emph{extension}, since 
in $\Rel$ it would be given by taking an image. 

Such restrictions and extensions, together with the product structure, do most 
of the querying work in \cref{section:querying}. \emph{Local products} make 
sense in this context and are discussed below in 
\cref{subsection:Appendixlocalproducts}. A `double category of relations' has a 
further property, however, namely, that of satisfying a \emph{discreteness 
axiom}, which ensures that its horizontal bicategory is in fact a compact 
closed monoidal category. This gives a good notion of partial map which is 
discussed below in \cref{subsection:AppendixPartialMaps}.

A double category $\mathbb D$ has \textbf{tabulators} if the external identity 
functor $\mathbb D_0\to \mathbb D_1$ has a right adjoint $\top$. This 
associates to each proarrow an object and a universal cell that we think of as 
a kind of \emph{comprehension scheme}. Tabulators provide type 
creation in double categorical ologs by associating a type to a proarrow of
those entities which satisfy the proposition represented by the given proarrow. 

Our data tables will be given by strict, cartesian double functors 
$M\colon\mathbb D\to \Rel$ valued in the double category of relations. Strict 
means that the associators and unitors are not mere isomorphisms, but are in 
fact equalities. Pseudo, hence strict, double functors preserve restriction and 
extension cells \cite[\S 6]{shulman2008}. Strict cartesian double functors 
$M\colon \mathbb D\to \Rel$ also preserve local products, as discussed 
immediately below. These preservation properties are used through \cref
{section:querying} of the paper and those calculations support our conceit that 
querying is functorial semantics.

\subsection{Local Products}
\label{subsection:Appendixlocalproducts}

Given two proarrows, say $R\colon X\proto Y$ and 
$S\colon X\proto Y$ in a cartesian equipment $\mathbb D$, the \textbf{local product} of $R$ and $S$ is defined to be the proarrow domain of the restriction 
along the diagonals
    \[\begin{tikzcd}
      X && Y \\
      {X\times X} && {Y\times Y}
      \arrow["\Delta"', from=1-1, to=2-1]
      \arrow[""{name=0, anchor=center, inner sep=0}, "{R\times S}"', "\shortmid"{marking}, from=2-1, to=2-3]
      \arrow[""{name=1, anchor=center, inner sep=0}, "{R\wedge S}", "\shortmid"{marking}, from=1-1, to=1-3]
      \arrow["\Delta", from=1-3, to=2-3]
      \arrow["{\text{restr}}"{description}, draw=none, from=1, to=0]
    \end{tikzcd}.\]
It computes a binary product in the hom-category $\mathbb D(X,Y)$. It's pretty easy to see that in $\Rel$, this local product is computed as a certain 
set of pairs. Namely, since it's a restriction, the resulting relation is the 
monic on the leftside of 
    \[\begin{tikzcd}
      {R\wedge S} && {R\times S} \\
      {X\times Y} && {X\times X\times Y\times Y}
      \arrow[from=1-1, to=2-1]
      \arrow["{\Delta_X\times \Delta_Y}"', "\shortmid"{marking}, from=2-1, to=2-3]
      \arrow[from=1-1, to=1-3]
      \arrow[from=1-3, to=2-3]
    \end{tikzcd}\]
where we've forgotten about one of the associativity isomorphisms in the 
product on the right. In any case, $R\wedge S$ in $\Rel$ is the set of pairs
    \begin{equation} \label{equation:localproductsinRelformula}
      R\wedge S = \lbrace (x,y)\mid xRy \text{ and } xSy\rbrace.
    \end{equation}
Thus, local products can be constructed by restricting products of relations
along diagonal arrows.
Alternatively, local products can viewed as a primitive notion with their
own double-categorical universal property \mbox{\cite[Example
  6.10]{patterson2024}}, and any cartesian equipment possesses finite local
products in this sense \mbox{\cite[Corollary 8.7]{patterson2024}}.

Inner joins can be constructed by adapting local products, which involves 
matching along a diagonal. That is, given two proarrows $R\colon A\times 
B\proto C$ and $S\colon B\proto D$, we define the \textbf{inner join of $R$ and 
$S$ along $B$} to be the restriction cell
    \[\begin{tikzcd}
      {A\times B} && {C\times D} \\
      {A\times B\times B} && {C\times D}
      \arrow["{1\times \Delta}"', from=1-1, to=2-1]
      \arrow[""{name=0, anchor=center, inner sep=0}, "{R\bowtie S}", "\shortmid"{marking}, from=1-1, to=1-3]
      \arrow["1", from=1-3, to=2-3]
      \arrow[""{name=1, anchor=center, inner sep=0}, "{R\times S}"', "\shortmid"{marking}, from=2-1, to=2-3]
      \arrow["{\text{restr}}"{description}, draw=none, from=0, to=1]
    \end{tikzcd}.\]
Think of this as giving two tables, namely, $R$ and $S$ with entries from sets 
$A$, $B$, $C$ and $D$. The tables have entries in one column in common, namely, 
those from the set $B$. We take the cartesian product of the relations $R\times 
S$ and then pull back along the diagonal morphism $\Delta$ on $B$ asking that 
those projection values are the same while leaving the others alone.

\subsection{Partial Maps}
\label{subsection:AppendixPartialMaps}

Any ordinary relation $R\colon A\proto B$ has a reverse relation 
$R^\dagger\colon B\to A$ which exchanges the domain and codomain. Suppose that 
$R\colon A\proto B$ is a relation whose associated span 
  \[\begin{tikzcd}
    R & B \\
    A
    \arrow["c", from=1-1, to=1-2]
    \arrow["d"', tail, from=1-1, to=2-1]
  \end{tikzcd}\]
is a \emph{partial map} in the sense that $d$ is monic. Define a morphism 
$\epsilon\colon R\times_A R \to B$ taking $(r,s)$ in the pullback of $d$ along 
itself to $cr = cs$, where equality is implied by the fact that $d$ is 
monic. This map $\epsilon$ defines a cell
  \[\begin{tikzcd}
    {R\times_AR} & {B\times B} \\
    B & {B\times B}
    \arrow["{\langle c,c\rangle}", from=1-1, to=1-2]
    \arrow["1", from=1-2, to=2-2]
    \arrow["\epsilon"', dashed, from=1-1, to=2-1]
    \arrow["\Delta"', from=2-1, to=2-2]
    \arrow["{\text{(I)}}"{description}, draw=none, from=1-1, to=2-2]
  \end{tikzcd}\]
which amounts to the condition that $R^\dagger\otimes R \leq \id_B$ holds in 
$\Rel$.

The converse is true. Consider any fixed relation $R\colon A\proto B$ and its 
reverse relation $R^\dagger\colon B\proto A$. Suppose that $R^\dagger\otimes R 
\leq \id_B$ holds. There is then is a cell $\epsilon \colon R^{\dagger}\otimes R 
\to B$ which is an arrow making a commutative square of the form of (I) 
above. The existence of $\epsilon$ forces $d$ to be monic. For if $r = 
(x,y)$ and $s = (x,z)$ are given elements of $R\times_A R$ (which is the same 
as supposing that $dr = ds$), we have that
  \begin{equation}
    \Delta \epsilon (r,s) = \langle c,c\rangle(r,s) = (y,z)
  \end{equation}
is a diagonal entry of $B\times B$ by (I) above meaning that $y=z$ must hold. 
Thus, the span consisting of $d$ and $c$ is a genuine partial map $A\to B$.

A general version for the necessity of the condition is valid in any unit-pure 
(the external identity functor is fully faithful) `double category of relations' 
with tabulators $\mathbb D$. Recall that in this case the horizontal bicategory 
$\mathbf H(\mathbb D)$ is then a cartesian bicategory 
\cite[Proposition 3.1]{lambert2022} and indeed a $\mathbf H(\mathbb D)$ is a 
`bicategory of relations', hence a compact closed monoidal category with 
involution on proarrows denoted by $(-)^\dagger$ \cite[Theorem 2.4]{carboni1987}. 
Recall from \cite[\S 2]{carboni1987} that a proarrow $r\colon A\proto B$ is a 
\textbf{partial map} if $r^\dagger\otimes r \leq \id_B$ holds.

\begin{theorem} \label{prop:PartialMapsAreMonic}
  If $\mathbb D$ is unit-pure and has tabulators, then $r\colon A\proto B$ is a
  partial map only if $d\colon \top r\to A$ is a monic arrow, meaning, in other
  words, that the span formed by the legs of $\top r$ is a partial map in the 
  sense that the domain arrow $d$ is monic.
\end{theorem}
\begin{proof}
  Suppose that $r$ is a partial map and suppose as given two morphisms $f,
  g\colon X\rightrightarrows \top r$ satisfying $df=dg$. Using the cell 
  $r^\dagger\otimes r\leq \id_B$ and ``unit-pure'' there is a unique morphism 
  $h$ that is equal to each of the legs in the outside of the figure on the 
  left in 
    \[\begin{tikzcd}
      {\top r\times_A\top r} & {\top r} & B && {\top r\times_A\top r} \\
      {\top r} & A \\
      B &&&& B
      \arrow[from=1-1, to=1-2]
      \arrow["d", from=1-2, to=2-2]
      \arrow[from=1-1, to=2-1]
      \arrow["d"', from=2-1, to=2-2]
      \arrow["\lrcorner"{anchor=center, pos=0.125}, draw=none, from=1-1, to=2-2]
      \arrow["c"', from=2-1, to=3-1]
      \arrow["c", from=1-2, to=1-3]
      \arrow["{\exists ! h}"', from=1-5, to=3-5]
    \end{tikzcd}\]
  in that $c\pi_1 = h = c\pi_2$ holds. Thus, $f$ and $g$ induce a pair morphism 
  $\langle f,g\rangle\colon X\to \top r\times_A\top r$, and we can calculate 
  that 
    \begin{equation}
      cf = c\pi_1\langle f,g\rangle = h\langle f,g\rangle = c\pi_2 \langle f,g\rangle = cg 
    \end{equation}
  meaning that $f = g$ must hold as a result of the uniqueness clause of the 
  universal property of the tabulator $\top r$. Therefore, $d$ is monic. 
\end{proof}

\end{document}